\documentclass[letterpaper, 10 pt, conference]{ieeeconf}  

\IEEEoverridecommandlockouts                              
\usepackage{cite}
\usepackage{amsmath,amssymb,amsfonts}
\usepackage{algorithmic}
\usepackage{graphicx}
\usepackage{textcomp}
\usepackage{amsmath,bm}
\usepackage{bbm}

\usepackage[font=singlespacing]{caption}
\usepackage{enumerate}

\usepackage{epsfig}
\usepackage{amsmath} 
\usepackage{amssymb}  
\usepackage{multirow}
\usepackage{bbm, dsfont}
\usepackage{bm} 
\usepackage{url}
\usepackage{balance}
\usepackage{color}
\graphicspath{{./figures/}}
\usepackage{array}
\usepackage[a]{esvect}

\newtheorem{theorem}{Theorem}
\newtheorem{definition}{Definition}

\newtheorem{lemma}{Lemma}

\newtheorem{remark}{Remark}
\newtheorem{proposition}{Proposition}

\newtheorem{problem}{Problem}

\usepackage[linesnumbered,ruled,vlined,onelanguage]{algorithm2e}

\newcommand{\yong}[1]{{\color{black} #1}}

\title{\LARGE \bf
Mesh-Based Affine Abstraction of Nonlinear Systems with Tighter Bounds
}

\author{Kanishka Raj Singh, Qiang Shen and Sze Zheng Yong
\thanks{The authors are with School for Engineering of Matter, Transport and Energy,
Arizona State University, Tempe, USA
       (email: {\tt \{kanishka.r.singh,qiang.shen,szyong\}@asu.edu})}%
       \thanks{This work was supported in part by DARPA grant D18AP00073. Toyota Research Institute (``TRI") also provided funds to assist the authors with their research but this article solely reflects the opinions and conclusions of its authors and not TRI or any other Toyota entity.}
}

\begin{document}

\maketitle

\begin{abstract}
In this paper, we consider the problem of piecewise affine abstraction of nonlinear systems, i.e., \yong{the over-approximation of its nonlinear dynamics by a pair of piecewise affine functions that ``includes'' the dynamical characteristics of the original system.} As such, guarantees for controllers or estimators based on the affine abstraction also apply to  the original nonlinear system. Our approach consists of solving a linear programming (LP) problem that over-approximates the nonlinear function at only the grid points of a mesh with a given resolution and then accounting for the entire domain via an appropriate correction term. To achieve a desired approximation accuracy, we also iteratively subdivide the domain into subregions.
 Our method applies to nonlinear functions with different degrees of smoothness, including Lipschitz continuous functions, and improves on existing approaches by enabling the use of tighter bounds. Finally, we compare the effectiveness of our approach with existing optimization-based methods in simulation and illustrate its applicability for estimator design.
\end{abstract}


\section{Introduction}
Abstraction-based methods for analyzing and controlling smart systems have recently attracted a great deal of interest \cite{Tabuada2009}. 
Since the dynamics of these smart systems, such as smart buildings, autonomous vehicles, and intelligent transportation, are almost always complex (nonlinear or hybrid), \yong{it is desirable to compute a simpler conservative approximation or abstraction that ``includes'' the dynamical characteristics of the original systems, for which (robust) controller/estimator designs may be easier than for the original complex systems.}
Using these approximate systems, controllers that are correct-by-construction with respect to reachability and safety specifications can be synthesized efficiently, see e.g., \cite{althoff2008reachability, Girard2012, Alimguzhin2017}, and similarly, guarantees for estimator designs also apply to the original complex systems \cite{Singh2018}. 

\emph{Literature Review.} The abstraction process typically involves partitioning the state space of the original system into a finite number of regions and approximating its dynamics locally in each region by a simpler dynamics, which is possibly conservative affine or polynomial approximations of the analyzed system \cite{Girard2012}.
When the system state moves from one region to another, the dynamics of the approximate system also switches accordingly.
That is to say, the approximate system behaves like a hybrid system and thus, the abstraction is also referred to as a hybridization process \cite{asarin2007hybridization,Asarin2003}.
In \cite{asarin2007hybridization,Asarin2003}, abstraction of nonlinear systems as piecewise linear systems over a mesh with a fixed partition was studied. 
However, this evenly-sized partition of the domain of interest may not be computationally tractable as it requires a large number of discrete states/modes to make the partition size sufficiently fine. 
To reduce the number of subregions, the Lebesgue piecewise affine approximation was proposed for a class of nonlinear Lipschitz continuous functions in \cite{Azuma2010}, where the partition of the state space depends on the variation of the vector field. 
On the other hand, on-the-fly abstraction is a dynamic method where the domain construction and the abstraction process are only carried out on states that are reachable \cite{Han2006, althoff2008reachability}.
Although this method scales better into high dimensions, some drawbacks, such as error accumulation and the splitting of the currently-tracked set of states along multiple facets, still exist \cite{Bak2016}. 

Abstractions are typically obtained by linear interpolation over a given region and adding the corresponding interpolation error to the simpler dynamics as bounded inputs 
\cite{asarin2007hybridization,Ramdani2009}.  Hence, a set of relevant literature pertains to the analysis of 
interpolation error bounds. 
The size of the error bounds is important as it affects not only the approximation precision but also the computation time. 
In \cite{Stampfle2000}, optimal estimates for approximation errors in linear interpolation of functions with several degrees of smoothness were developed, while \cite{dang2010accurate} presented a coordinate transformation to get a tighter interpolation error bound. 

\emph{Contributions.} In this paper, we propose a novel piecewise affine abstraction method that over-approximates nonlinear dynamics.  
Specifically, we bracket the original nonlinear dynamics in each bounded subregion of the state space by two piecewise affine functions instead of only having an interval-valued affine vector, in contrast to the hybridization approaches in \cite{asarin2007hybridization, dang2010accurate, Girard2012}. Moreover, we develop a mesh-based method for piecewise affine abstraction, which over-approximates the nonlinear behaviors over an entire mesh as opposed to over each simplex/mesh element, thus our approach results in less complex abstractions that can simplify reachability analysis.
The novelty of our approach lies in solving a linear programming (LP) optimization that over-approximates the nonlinear function at only the grid points of a mesh with a given resolution and then accounting for the entire domain in the interior of the mesh via an appropriate correction term.  
The proposed abstraction algorithm can also obtain an arbitrarily precise approximation of a nonlinear function at the price of increasing the mesh resolution, hence the size of the LP and its computational complexity. 

Comparing with a recent abstraction method for Lipschitz continuous functions in \cite{Alimguzhin2017}, our method can apply to nonlinear functions with different degrees of smoothness including Lipschitz continuous functions.
In addition, our analysis is based on mesh elements (in contrast to point-wise analysis in \cite{Alimguzhin2017}) and this enables the use of tighter error bounds based on linear interpolation in \cite{Stampfle2000,dang2010accurate}. 
Therefore, the abstraction efficiency is improved and the number of subregions is reduced for the same desired approximation accuracy. 
Finally, using simulation examples, we demonstrate the advantages of the proposed approach over existing optimization-based approaches in \cite{Alimguzhin2017} and our prior approach using mixed-integer nonlinear programming (MINLP) \cite{Singh2018}, as well as illustrate the usefulness of the obtained abstraction for estimator design, specifically the active model discrimination problem.

\section{Background} \label{sec_background}


\subsection{Notation}

For a vector $v \in \mathbb{R}^n$ and a matrix $M \in \mathbb{R}^{p \times q}$, $\|v\|_i$ and $\|M\|_i$ denote their  (induced) $i$-norm with $i=\{1,2,\infty\}$.
$[n]$ is an initial segment ${1,2, \ldots, n}$ of the natural numbers.

\subsection{Modeling Framework and Definitions} \label{sec:model}

Consider a nonlinear system $\mathcal{G}$ described by 
\begin{gather}
\begin{array}{c}
x^{+} = f(x,u), \label{nonl_sys}\\
\text{subject to } g(x,u) \leq 0, 
\end{array}
\end{gather}
where $x \in \mathcal{X}$ is the system state at the current time instant with a closed interval domain $\mathcal{X} = [a_{x}, b_{x}]^n \subset \mathbb{R}^{n}$, 
$ u \in \mathcal{U}$ is the control input with a closed interval domain $\mathcal{U} = [a_{u}, b_{u}]^m\subset \mathbb{R}^{m}$ 
and $f:\mathcal{X}\times \mathcal{U} \to \mathbb{R}^n$, $g:\mathcal{X}\times \mathcal{U}  \to \mathbb{R}^q$ are continuous vector fields (that belong to several smoothness classes). 
Specifically, we consider $f$ and $g$ that are Lipschitz continuous (with constant $\lambda$), $C^0$, $C^1$, and $C^2$ functions.
For discrete-time systems, $x^{+}$ denotes the state at the next time instant while for continuous-time systems, $x^{+}=\dot{x}$ is the time derivative of the state. Moreover, we define a \emph{cover} of the compact state-input domain $\mathcal{X}\times \mathcal{U}  \subseteq \mathbb{R}^{n+m}$, where the domain is divided into $p$ subregions that constitute its \textit{cover}:

\begin{definition}[Cover] \label{def_cover}
A \textit{cover} $\mathcal{I}$ of the closed bounded region $\mathcal{X} \times \mathcal{U}  \subset \mathbb{R}^{n+m}$ is a collection of $p$ subregions $\mathcal{I} = \{ I_i | i \in [p]\} $ such that $\mathcal{X} \times \mathcal{U}\subseteq \bigcup_{i =1}^{p} I_i$.
\end{definition}

Note the cover can be different for each dimension of the vector-valued $f$. However, for simplicity and without loss of generality, we assume that the cover is the same for all elements of $f$, and similarly, for the vector-valued $g$.
In particular, we will consider a cover whose subregions are uniform meshes, defined as:

\begin{definition}[Uniform Mesh] \label{mesh} 
A \textit{uniform mesh} of each subregion $I_i\subseteq \mathcal{X} \times \mathcal{U}$ is a collection of partitions, called mesh elements, with $r_i$ grid points along each direction/dimension $j$, for all $j\in[n+m]$. 
$r_i$ is considered as the resolution for the mesh in each subregion $I_i$ (or $r$ only if $r_i$ is the same for all subregions). 
\end{definition}

Moreover, we define the diameter of a polytope as:
\begin{definition}[Diameter] \label{diameter} 
The \textit{diameter} $\delta$ of a polytope is the greatest distance between two vertices of the polytope.
\end{definition}

\begin{remark} \label{remark_mesh} 
The uniform mesh for each subregion in this paper is a hyperrectangular mesh generated by a uniform grid of $r$ mesh points in every direction/dimension, where each mesh element is itself a hyperrectangle.
Based on this uniform hyperrectangular mesh element, we can easily obtain  simplicial mesh elements.
Consider the hyperrectangle $[a,b]^{n+m} $ and let $\Theta$ be the set of permutation of $[n+m]$.  
For all $\theta = (j_1, \ldots, j_{n+m}) \in \Theta$, the set $\mathcal{S}_{\theta} = \{ z \in [a,b]^{n+m}: a \le z_{j_1} \le \cdots \le z_{j_{n+m}} \le b \}$ is a simplex of $\mathbb{R}^{n+m}$. 
A proof of this can be found in \cite{Kuhn1960}.
Based on the Definition \ref{diameter}, it is easy to verify that the resulting simplicial mesh elements have the same diameter and vertices as the hyperrectangular mesh element, a fact that we will use in Lemma \ref{lemma}.
\end{remark}


Then, for each subregion $I_i \in \mathcal{I}$ that covers the domain of interest, our goal is to over-approximate/abstract the nonlinear  $f$ by a pair of affine functions $\underline{f}_i$ and $\overline{f}_i$ such that for all $(x,u) \in I_i$, we have that $\underline{f}_i(x,u) \le f(x,u) \le \overline{f}_i(x,u)$. 
These affine functions with respect to $f$ over $I_i \in \mathcal{I}$ are 
\begin{align}
\underline{f}_i(x,u) = \underline{A}_i x + \underline{B}_i u + \underline{h}_i, \\
\overline{f}_i(x,u) = \overline{A}_i x + \overline{B}_i u + \overline{h}_i,
\end{align}
where the matrices $\underline{A}_i$, $\overline{A}_i$, $\underline{B}_i$, $\overline{B}_i$, and the vectors $\underline{h}_i$ and $\overline{h}_i$ are constant and of appropriate dimensions. 
Let $(\underline{\mathcal{F}}, \overline{\mathcal{F}})$ be a pair of families of affine functions with $\underline{\mathcal{F}} = \{\underline{f}_1, \ldots, \underline{f}_p\}$ and $\overline{\mathcal{F}} = \{\overline{f}_1, \ldots, \overline{f}_p\}$. 
Then, the nonlinear function $f: \mathcal{X} \times \mathcal{U} \rightarrow \mathbb{R}^n$ is over-approximated with a pair of affine families $(\underline{\mathcal{F}}, \overline{\mathcal{F}})$ over a cover $\mathcal{I}$ (i.e., a pair of piecewise affine functions) if $\underline{f}_i(x,u) \le f(x,u) \le \overline{f}_i(x,u)$, $\forall i\in[p]$ and $\forall (x,u) \in I_i$. 

In the same way, we over-approximate/abstract the nonlinear constraint $g$ by a piecewise affine function $\overline{g}_i$ such that for all $(x,u) \in I_i$, we have that $g(x,u) \le \overline{g}_i(x,u)  \le 0$. 
As before, define affine functions $\underline{g}_i$ and $\overline{g}_i$ over $I_i \in \mathcal{I}$ as 
\begin{align}
\underline{g}_i(x,u) = \underline{C}_i x + \underline{D}_i u + \underline{w}_i,\\
\overline{g}_i(x,u) = \overline{C}_i x + \overline{D}_i u + \overline{w}_i,
\end{align}
where the matrices  $\underline{C}_i$, $\overline{C}_i$, $\underline{D}_i$, $\overline{D}_i$, and the vectors $\underline{w}_i$, $\overline{w}_i$  are constant and of appropriate dimensions. 
Let $(\underline{\mathcal{G}}, \overline{\mathcal{G}})$ be a pair of families of affine functions with $\underline{\mathcal{G}} = \{\underline{g}_1, \ldots, \underline{g}_p\}$ and $\overline{\mathcal{G}} = \{\overline{g}_1, \ldots, \overline{g}_p\}$. 
The nonlinear constraint function $g(x,u) \le 0$ can be over-approximated with an affine family $\overline{\mathcal{G}}$  over a cover $\mathcal{I}$ (i.e., a piecewise affine function) if $ g(x,u) \le \overline{g}_i(x,u)  \le 0$, $\forall i\in[p]$ and $\forall (x,u) \in I_i$. 


Note that the \emph{lower} affine family $\underline{\mathcal{G}}$ is not part of the abstraction but is needed for the definition of approximation error below, which will be used as the objective function for our LP problem in Theorem \ref{theorem1}.

\begin{definition}[Approximation Error] \label{approx}
Consider 
a cover $\mathcal{I} = \{ I_i | i \in [p]\}$ of $\mathcal{X} \times \mathcal{U} \subset \mathbb{R}^{n+m}$. 
If a pair of affine families $(\underline{\mathcal{F}}, \overline{\mathcal{F}})$ over-approximate a nonlinear function $f$ over the cover $\mathcal{I}$, then the approximation error with respect to the nonlinear dynamics is defined as $e(\underline{\mathcal{F}}, \overline{\mathcal{F}}) = \max_{i\in[p]} \max_{(x,u) \in I_i} \|\overline{f}_i(x,u) - \underline{f}_i(x,u)\|_\infty$. 
Similarly, if a pair of affine families $(\underline{\mathcal{G}}, \overline{\mathcal{G}})$ over-approximate the nonlinear constraint $g(x,u)\le 0$ over the cover $\mathcal{I}$, then the approximation error with respect to the nonlinear constraint is defined as $e(\underline{\mathcal{G}}, \overline{\mathcal{G}}) = \max_{i\in[p]} \max_{(x,u) \in I_i} \| \overline{g}_i(x,u) - \underline{g}_i(x,u)\|_\infty$. 
\end{definition}

\section{Problem Formulation}

The piecewise affine abstraction for the nonlinear system $\mathcal{G}$ in \eqref{nonl_sys} consists of the following two abstraction problems:

\begin{problem}[Affine Abstraction of Nonlinear Dynamics] \label{problem1}
For a given nonlinear $n$-dimensional vector field $f(x,u)$ with $(x,u) \in {\mathcal{X}} \times {\mathcal{U}}$ as defined in Section \ref{sec:model} and a given desired accuracy $\varepsilon_f$, find a cover $\mathcal{I} = \{I_{1} ,\ldots\,,I_{p}\}$ and a pair of $n$-dimensional family of affine hyperplanes $\overline{\mathcal{F}}\,=\{\overline{f}_{1}, \ldots\,,\overline{f}_{p}\} $ and $\underline{\mathcal{F}}\,=\{\underline{f}_{1}, \ldots\,,\underline{f}_{p}\}$  such that: 
\begin{align}\label{eq:prob1}
\hspace*{-0.15cm}\begin{array}{l}
e(\underline{\mathcal{F}},\overline{\mathcal{F}})  \leq \varepsilon_f,\\ 
\underline{f}_{i}(x,u) \leq f(x,u) \leq \overline{f}_{i}(x,u), \, \,  \forall (x,u) \in \mathcal{I}_{i}, \forall i \in [p],\end{array}\hspace*{-0.15cm}
\end{align}
where $e(\underline{\mathcal{F}},\overline{\mathcal{F}})$ is the approximation error (see Definition \ref{approx}). The pair of affine families $(\underline{\mathcal{F}},\overline{\mathcal{F}})$ is then the abstracted model (i.e., affine abstraction) of the nonlinear dynamics.
\end{problem}
\vspace{0.2cm}
\begin{problem}[Affine Abstraction of Nonlinear Constraints] \label{problem2}
For a given nonlinear $q$-dimensional constraint $g(x,u) \leq 0 $ with $(x,u) \in {\mathcal{X}} \times {\mathcal{U}}$ as defined in Section \ref{sec:model} and a desired accuracy of $\varepsilon_g$, find a cover $\mathcal{I} = \{I_{1}, \ldots\,,I_{p}\}$ and a pair of $q$-dimensional family of affine hyperplanes $\underline{\mathcal{G}}\,=\{\underline{g}_{1},\ldots\,,\underline{g}_{p}\} $ and $\overline{\mathcal{G}}\,=\{\overline{g}_{1},\ldots\,, \overline{g}_{p}\} $  such that: 
\begin{align}
\hspace*{-0.3cm}\begin{array}{l}
e(\underline{\mathcal{G}}, \overline{\mathcal{G}})  \leq \varepsilon_g,\\
\hspace*{-0.05cm} \underline{g}_{i}(x,u) \leq 
g(x,u) \leq \overline{g}_{i}(x,u)  \leq 0, 
  \forall (x,u) \hspace{-0.05cm}\in\hspace{-0.05cm} \mathcal{I}_{i}, \forall i \hspace{-0.05cm}\in \hspace{-0.05cm} [p], \hspace{-0.15cm}
  \end{array}\hspace*{-0.25cm}
\end{align}
where $e(\underline{\mathcal{G}}, \overline{\mathcal{G}})$ is the approximation error (see Definition \ref{approx}).\\ The affine constraints $\overline{g}_i \le 0, \forall \overline{g}_i \in \overline{\mathcal{G}}$ are then the abstracted model (i.e., affine abstraction) of the nonlinear constraint.
\end{problem}
\section{Main Results on Affine Abstraction}

In this section, we will mainly focus on addressing Problem \ref{problem1}, since the same approach also directly applies to Problem \ref{problem2}.
There are two parts in solving the problem. 

In the first part, we consider the subproblem of abstracting a single pair of affine hyperplanes for the nonlinear dynamics in a single subregion $I_i \in \mathcal{I}$ using mesh-based affine abstraction. 
Unlike the recent paper \cite{Alimguzhin2017} in which only Lipchitz continuous functions have been considered, we provide a novel analysis that considers mesh elements, as opposed to point-wise analysis, which enables us to exploit the tighter bounds from the literature on linear interpolation  \cite{Stampfle2000,dang2010accurate} for several classes of continuous functions with different degrees of smoothness. 

Then, in the second subproblem, we extend the abstraction method from a single subregion to multiple subregions, which constitute a cover of the state space of the nonlinear dynamics. 
Specifically, we will construct an \emph{$\varepsilon_f$-accurate} cover that is composed of subregions with a pair of families of affine hyperplanes $(\overline{\mathcal{F}}, \underline{\mathcal{F}})$ such that the nonlinear dynamics $f$ is over-approximated with desired accuracy $\varepsilon_f$, i.e., \eqref{eq:prob1} holds in each subregion.

As will be demonstrated in Section \ref{sec:1d}, our abstraction method outperforms the algorithm in \cite{Alimguzhin2017} for a given $\varepsilon_f$ in terms of  computation time and number of subregions required to over-approximate a function. 


\subsection{Mesh-Based Affine Abstraction of a Single Subregion}

To solve the subproblem of mesh-based affine abstraction of a single subregion, we will rely on the following result on linear interpolation error bounds over simplices:
\begin{proposition}[\hspace{-0.02cm}{\cite[Theorem 4.1  \& Lemma 4.3]{Stampfle2000}}] \label{proposition}
Let $S$ be an $(n+m)$-dimensional simplex such that $S \subseteq \mathbb{R}^{n+m}$ 
with  diameter $\delta$ {(see Definition \ref{diameter})}. Let $f:S \rightarrow \mathbb{R}$ be a nonlinear function  and  let $f_{l}$ be the linear interpolation of $f$ at the vertices of the simplex $S$. Then, the approximation error bound $\sigma$  defined as the maximum error between $f$ and $f_{l}$ on $S$:
\begin{align}
\sigma = \max_{s \, \in S} (|f(s) - f_{l}(s)|) \label{sigma}
\end{align}
is upper-bounded by 
\begin{enumerate}[(i)]
\item $\sigma \leq 2\lambda \delta_{s}$, if $f \in C^{0}$ on $S$,
\item $\sigma \leq \lambda \delta_{s}$, if $f$ is Lipschitz continuous on $S$,
\item $\sigma \leq \delta_{s} \max_{s\in S}\|f'(s)\|_{2}$, if $f \in C^{1}$ on $S$,
\item $\sigma \leq \frac{1}{2} \delta_{s}^{2} \max_{s\in S} \|f''(s)\|_{2}$, if $f \in C^{2}$ on $S$,
\end{enumerate}
where $\lambda$ is the Lipschitz constant, $f'(s)$ is the Jacobian of $f(s)$, $f''(s)$ is the Hessian of $f(s)$ and $\delta_{s}$ is simplex ball radius that satisfies
\begin{align*}
\delta_{s} \leq \sqrt{\frac{n+m}{2(n+m+1)}} \delta.
\end{align*}
\end{proposition}

According to \cite{Stampfle2000}, all factors are the best possible, while \cite{dang2010accurate} proposes a mapping of the original simplex to
to an ``isotropic" space to obtain a better bound for the simplex ball radius $\delta_s$. On the other hand, the Lipschitz constant $\lambda$ for $f$ on $S$ can be computed using well-known techniques, e.g., \cite{mladineo1986algorithm}, while the constants for cases (iii) and (iv) above can be computed using any off-the-shelf optimization software.

Moreover, we derive a useful lemma as follows:
\begin{lemma}\label{lem:1}
Let $f_1$ and $f_2$ be affine hyperplanes on the same $(n+m)$-dimensional simplicial domain $S_k\subseteq \mathbb{R}^{n+m}$ with vertex set $\mathcal{V}_k = \{v_{k_1},\ldots,v_{k_{n+m+1}}\} $. Suppose that 
\begin{align}
f_1(v_{k_i}) &\geq f_2(v_{k_i}), \, \, \forall \, \, i \in [n+m+1]. \label{proofeqlem1}
\end{align}
Then, 
$f_1(s) \geq f_2(s), \, \, \forall \, s  \in \, S$.
\end{lemma}

\begin{proof}
Since $S$ is a simplex, any point $s \in S$ can be represented as $s=\sum_{i=1}^{n+m+1} \alpha_i v_{k_i}$, where $\alpha_i \geq 0$, $\sum_{i=1}^{n+m+1} \alpha_i=1$. Moreover, we represent the affine hyperplanes as 
\begin{align*}
f_1(s)&=A_1s+b_1=\hspace{-0.2cm}\sum_{i=1}^{n+m+1} \hspace{-0.2cm}\alpha_i (A_1 v_{k_i}+b_1) =\hspace{-0.2cm}\sum_{i=1}^{n+m+1} \hspace{-0.2cm}\alpha_i f_1(v_{k_i}),\\
f_2(s)&=A_2 s +b_2=\hspace{-0.2cm}\sum_{i=1}^{n+m+1}\hspace{-0.2cm} \alpha_i (A_2 v_{k_i}+b_2)=\hspace{-0.2cm}\sum_{i=1}^{n+m+1} \hspace{-0.2cm}\alpha_i f_2(v_{k_i}).
\end{align*}
Since \eqref{proofeqlem1} holds by assumption and $\alpha_i \geq 0$, the result follows directly from the above.
\end{proof}

Armed with the above interpolation error bounds and lemma, we can obtain the following lemma and theorem using a novel analysis that considers mesh elements for each subregion, as opposed to point-wise analysis in \cite{Alimguzhin2017}, resulting in tighter bounds and more effective abstraction.

\begin{lemma} \label{lemma}

Given a nonlinear function $f: I \rightarrow \mathbb{R}^n$ with a hyperrectangular domain $I \subset \mathbb{R}^{n+m}$ for any subregion $I \in \mathcal{I}$, let $\mathcal{V} = \{{v}_{1}, {v}_{2}, \hdots, {v}_{l} \}$ be a set of $l$ grid points of a uniform mesh of the subregion $I$ (see Definition \ref{mesh}). 
Suppose that we have affine hyperplanes $f_{u}$ and $f_{b}$ such that:
\begin{align}  
f_{u}({v}_{i}) &\geq f({v}_{i}), \, \forall i \in [l] \label{equation_f_{u}}, \\
f_{b}({v}_{i}) &\leq f({v}_{i}), \, \forall i \in [l] \label{equation_f_{b}},
\end{align}
then, the affine hyperplanes $\overline{f}$ and $\underline{f}$ over-approximate the function $f$ in the entire subregion $I$, i.e., 
\begin{align}
\overline{f}(x,u) &= f_{u}(x,u) + \sigma \geq f(x,u), \, \,\forall (x,u) \in I \label{equation_fbar1},\\
\underline{f}(x,u) &= f_{b}(x,u) - \sigma \leq f(x,u),  \, \,\forall (x,u) \in I, \label{equation_fbar2}
\end{align}
where $\sigma$ is a vector of the smallest possible error bounds based on the degrees of smoothness of each element of the vector-valued function $f$ (see Proposition \ref{proposition}).
\end{lemma}

\begin{proof}
First, we note that the given hyperrectangular mesh can be considered to be comprised of simplices with the same set of vertices as described in {Remark \ref{remark_mesh}}. Next, consider any $(n+m)$-dimensional simplex $S_k \subset I$ with vertex set $\mathcal{V}_k = \{v_{k_1},\ldots,v_{k_{n+m+1}}\} $. By assumption, there exists an affine plane $f_{u}$ that satisfies 
\eqref{equation_f_{u}}, and hence also at the vertices in $\mathcal{V}_k$, i.e.,  
\begin{align}
f_{u}(v_{k_i}) &\geq f(v_{k_i}), \, \, \forall \, \, i \in [n+m+1], \label{proofeq1}
\end{align}
since $\mathcal{V}_k \subseteq \mathcal{V}$. Moreover, the linear interpolation of the simplex vertices, $f_{l}(x,u), \, \forall \, (x,u) \in S_k$ is a uniquely determined affine plane. Since $f_{u}$ and $f_{l}$ are both affine over the same domain, by Lemma \ref{lem:1}, we have
\begin{align*}
f_{u}(x,u) &\geq f_{l}(x,u),  &\forall  (x,u)  \in  S_k, \\
\implies  \overline{f}(x,u)=f_{u}(x,u) + \sigma &\geq f_{l}(x,u) + \sigma, \hspace{-0.2cm} &\forall  (x,u)  \in  S_k. 
\end{align*}
By Proposition \ref{proposition},  $f_{l}(x,u) + \sigma \hspace{-0.05cm} \ge \hspace{-0.05cm}  f(x,u), \forall  (x,u) \in S_k$, hence  
\begin{align*}
\overline{f}(x,u) &\geq f(x,u), \, \, \forall \, \, (x,u) \in S_k.
\end{align*}
Since this result is applicable for all $S_k \subseteq I$ with the same $f_{u}$, we further have
\begin{align}
\overline{f}(x,u) \geq f(x,u), \, \, \forall \,  (x,u) \in I.
\end{align}
A similar proof can be derived to obtain \eqref{equation_fbar2}.
\end{proof}

\begin{theorem} \label{theorem1}
Given a nonlinear function $f: I\rightarrow \mathbb{R}^n$ with a hyperrectangular domain $I \subset \mathbb{R}^{n+m}$ for any subregion $I \in \mathcal{I}$, let $\mathcal{V} = \{{v}_{1}, {v}_{2}, \hdots, {v}_{l} \}$ be a set of $l$ grid points of a uniform mesh of the subregion $I$ (see Definition \ref{mesh}) and  $\mathcal{C} =  \{ {v}^{c}_{1}, \ldots, {v}^{c}_{2^{(n+m)}} \}$ be a set of the corner points of the hyperrectangular domain.
The affine hyperplanes $\overline{f}$ and $\underline{f}$ that over-approximate/abstract $f$ are given by:
\begin{align*}
\overline{f} = f_{u} + \sigma, \quad 
\underline{f} = f_{b} - \sigma,
\end{align*}
with $\sigma$ as defined in Lemma \ref{lemma}, 
$f_{u} = \overline{A}\,x + \overline{B} \, u + h_{u}$, and 
$f_{b} = \underline{A}\,x + \underline{B} \, u + h_{b}$, 
where $\overline{A}, \underline{A}, \overline{B}, \underline{B}, h_{u}$ and $h_{b}$ are obtained from the following linear programming (LP) problem:
\begin{subequations} 
\begin{align}
\nonumber &\underset{\theta, \overline{A}, \underline{A}, \overline{B}, \underline{B}, h_{u},h_{b}}{\text{min}}
 \theta \\[-0.2em]
&\text{subject to} \quad 
\overline{A}\, {x}_{i} + \overline{B} \, {u}_{i} + h_{u} \geq f({x}_{i},{u}_{i}), \label{eq:a}\\[-0.2em]
& \hspace{1.75cm} \underline{A}\, {x}_{i} + \underline{B} \, {u}_{i} + h_{b} \leq f({x}_{i},{u}_{i}), \label{eq:b}\\[-0.2em]
& \hspace{0.7cm} (\overline{A} - \underline{A}) \, {x}^{c}_{j}+(\overline{B} - \underline{B}) \, {u}^{c}_{j} + h_{u} - h_{b} \leq \theta \mathds{1}_n, \label{eq:c} \\[-0.2em]
\nonumber & \hspace{1.75cm} \forall i \in [l], \forall \, j\in [2^{(n+m)}], 
\end{align}
\end{subequations}
where $\mathds{1}_n$ represents the $n$-dimensional vector of ones, $(x_i, u_i)$ and $(x^c_j, u^c_j)$ are the state-input values at the grid point $v_i$ of the mesh and the vertex $c_j$ of $I$, respectively.
\end{theorem}

\begin{proof}
The first two constraints \eqref{eq:a} and \eqref{eq:b} in the linear optimization problem can be interpreted as: 
\begin{align*}
\overline{A}\,  {x}_{i} + \overline{B} \,  {u}_{i} + h_{u} = f_{u}( {v}_{i}) \geq f( {v}_{i}), \, \forall i \in [l],\\
\underline{A}\,  {x}_{i} + \underline{B} \,  {u}_{i} + h_{b} = f_{b}( {v}_{i}) \leq f( {v}_{i}), \, \forall i \in [l].\end{align*}
Based on {Lemma \ref{lemma}}, these inequalities imply that
\begin{align*}
\overline{f}(x,u) \geq f(x,u) \,, \forall (x,u) \, \in I,  \\
\underline{f}(x,u) \leq f(x,u) \,, \forall (x,u) \, \in I, 
\end{align*}
which means that \eqref{eq:a} and \eqref{eq:b} always make sure that $\overline{f}$ and $\underline{f}$ are completely over and under $f$ in $I$, as required by the definition of affine abstraction.  
Next, 
we wish to make $f_{u}$ and $f_{b}$ to be as close to each other as possible by minimizing $\theta$, defined as:
\begin{align*}
\theta = \max_{(x,u) \in \mathcal{X} \times \mathcal{U}} \|f_{u}(x,u) - f_{b}(x,u)\|_\infty.
\end{align*}

We now show that this can be rewritten as a minimization problem with the objective function $\theta$ and the third constraint \eqref{eq:c}. 
Consider any one dimension in $\mathbb{R}^{n+m}$ with the other dimensions arbitrarily fixed. Due to the linear nature of the difference between $f_{u}$ and $f_{b}$, the difference can only be increasing or decreasing as the considered point in $I$ moves in one direction. Because of this, the maximum difference would be at one of the ends. Since this argument applies to all dimensions, it follows that the maximum difference must be attained at one of the vertices of $I$. Hence, we only need to minimize the difference among the vertices of the $(n+m)$-dimensional hyperrectangle  $I$, which leads to the third constraint \eqref{eq:c}.
\end{proof}

%
%
%
%
\subsection{Mesh-Based Affine Abstraction of Multiple Subregions} \label{subsection2}

For multiple subregions, the mesh-based affine abstraction is provided in Algorithm \ref{algorithm1}, in which the abstraction method of a single subregion (see Theorem \ref{theorem1}) is considered as the \texttt{abstraction} function.
In Algorithm \ref{algorithm1}, the {\texttt{epsCover}} function is recursive in nature. 
First, the {\texttt{abstraction}} function is run in order to obtain $\overline{f},\underline{f}$ and $e(\overline{f},\underline{f})$. 
Then, the error $e(\overline{f},\underline{f})$ is compared to the desired error $\varepsilon_f$. 
If it is smaller than $\varepsilon_f$, the information about the subregion boundary ({$\tt{bound}$}) and the corresponding hyperplanes (with desired accuracy) is collected in a data structure called {$\tt{cover}$}. 
Otherwise, the function {\texttt{divBound}} divides the state domain into a finer cover $\mathcal{I} = \{I_{1}, \ldots, I_{2^{n+m}}\}$ by partitioning each interval $[a_j, b_j], \forall j\in [n+m]$ into two subintervals of width $(b_j-a_j)/2$. Thus, the region is divided into $2^{(n+m)}$ different subregions denoted by {$\tt{subBounds}$}. 
Now, each of the subregions {$\tt{subBounds}$} is recursively passed to {\texttt{epsCover}} in place of the original region until $e(\overline{f},\underline{f})$ in each newly obtained subregion has an error that is less than $\varepsilon_f$. 
In each recursion, we keep tracking of the subregion boundaries and the corresponding affine-hyperplanes and store it in the data structure {$\tt{cover}$}.

\begin{algorithm}[tp]
  \SetAlgoLined\DontPrintSemicolon
\KwData{$f$, $\tt{bound} = \mathcal{X} \times \mathcal{U}$, resolution $r$, desired accuracy $\varepsilon_f$}		
 
\SetKwFunction{abs}{epsCover}
\SetKwFunction{lin}{abstraction}
\SetKwFunction{algo}{divBounds}

\SetKwProg{func}{function}{}{}

\func{\abs{$f,{\tt{bound}}, r, \varepsilon_f$}}
{  
	$(\overline{f},\underline{f}, e(\overline{f},\underline{f}))$   $\leftarrow$ \lin{$f,{\tt{bound}}, r, \varepsilon_f$} \;
	
   \eIf{$e(\overline{f},\underline{f}) \leq \varepsilon_f$}
   { 	
   ${\tt{cover}} = \{\overline{f}, \underline{f}, {\tt{bound}}\}$\\
   \KwRet $(\tt{cover})$ \\
   }
   { 	
   $\mathcal{I} \leftarrow$ \algo{$\tt{bound}$} \\
   \For{$ i = 1: 2^{n+m}$} 
	{ {${\tt{cell}}\{i\}$ = \abs{$f,I_i,r,\varepsilon_f$}}
	}
	${\tt{cover}} = \bigoplus_{i=1}^{2^{n+m}}\hspace{-0.1cm} \{{\tt{cell}}\{i\}\}$  ($\bigoplus$\,=\,concatenation) \\
   }
   \KwRet {$({\tt{cover}}, \mathcal{I})$} \;
}
\setcounter{AlgoLine}{0}
\func{\algo{$\tt{bound}$}}{
	Refer to Section \ref{subsection2} for its description\; 
	\KwRet{$(\tt{subBounds})$}
}
\setcounter{AlgoLine}{0}
\func{\lin{$f,{\tt{bound}}, r, \varepsilon_f$}}
{
	Refer to \emph{Theorem \ref{theorem1}} for its description\;
	\KwRet{$(\overline{f},\underline{f}, e(\overline{f},\underline{f}))$}
}
  \caption{Creating a $\varepsilon_f$-accurate Cover}
  \label{algorithm1}
\end{algorithm}

\section{Simulation Examples}\label{sec:examples}

In this section, we investigate the effects of the choices of various parameters on the proposed mesh-based affine abstraction algorithm. In particular, we consider the impacts of the desired accuracy $\varepsilon_f$ and approximation error bound $\sigma$ in Section \ref{sec:1d} and the resolution vector $r$ in Section \ref{sec:2d}. In addition, we compare our approach with the algorithm in \cite{Alimguzhin2017} in Section \ref{sec:1d} and with the MINLP approach from our previous work \cite{Singh2018} in Section \ref{sec:2d}.
All the examples are implemented in MATLAB on a 2.9 GHz Intel Core i5 CPU.

\subsection{One-Dimensional Example ($f(x,y)=x \cos y$)} \label{sec:1d}

In order to compare the effectiveness of our affine abstraction approach with that in \cite{Alimguzhin2017}, we begin by applying our algorithm to the same one-dimensional nonlinear function $f(x,y)=x \cos y$, on the interval $[-2,2] \times [0,2\pi]$. Since this function is infinitely differentiable, all approximation error bounds $\sigma$ from Proposition \ref{proposition} apply and these bounds are used to obtain Table \ref{table:numerical} for three different desired accuracies, $ \varepsilon_f \in \{0.05,0.1,0.2\}$. The resulting number of subregions serve as a measure for quality of the abstraction procedure because a better approximation would naturally lead to fewer subregions that are required for obtaining a given desired accuracy $ \varepsilon_f$ (cf. Figure \ref{fig:abstraction}).

\begin{table}[h!]
\begin{minipage}{0.485\textwidth}
	\centering
	\caption{Results of affine abstraction for the nonlinear function $x \cos y$ for varying desired accuracies $ \varepsilon_f$ and varying approximation error bounds $\sigma$ (shown for the entire domain) corresponding to different degrees of smoothness. } \vspace{-0.1cm}
	\label{table:numerical}
	\begin{tabular}{l l c c c }
		\hline
		 & Desired Accuracy, $\varepsilon_f$ & $0.2$ & $0.1$ & $0.05$ \\ \hline \\[-1em]
		\multirow{2}{*}{\begin{tabular}[l]{@{}l@{}}(i) $C^0$ function\\ \ \ \ \ ($\sigma=1.351$)\end{tabular} }                                                                  & No. of Subregions   & 784  &  1024   & 4096 \\ 
		& CPU Time (s) & 169.97  &   212.48  & 765.22  \\ \hline \\[-1em]
		\multirow{2}{*}{\begin{tabular}[c]{@{}l@{}}(ii) Lipschitz function \\ \quad \ \ ($\sigma=0.676$)\end{tabular}}                                                                  & No. of Subregions   &  256 &  976  &  3376\\ 
		& CPU Time (s) & 55.81   &  213.15   &   674.49\\ \hline \\[-1em]
		\multirow{2}{*}{\begin{tabular}[c]{@{}l@{}}(iii) $C^1$ function \\ \qquad ($\sigma=0.478$)\end{tabular}}                                                                  & No. of Subregions   & 232  &  688  & 1024 \\ 
		& CPU Time (s) &  50.84  &  149.83   &  212.89 \\ \hline \\[-1em]
		\multirow{2}{*}{\begin{tabular}[c]{@{}l@{}}(iv) $C^2$ function  \\ \qquad($\sigma=0.228$)\end{tabular}}                                                                  & No. of Subregions   & 64  & 232   & 256 \\ 
		& CPU Time (s) & 14.22   &  50.77   &  56.83 \\ \hline\hline \\[-1em]
				\multirow{2}{*}{\begin{tabular}[c]{@{}l@{}}\cite{Alimguzhin2017}\footnote{Note that the approximation error bound $\sigma$ is defined differently in \cite{Alimguzhin2017}, where the error is added \emph{before} the optimization routine is executed, unlike our approach that adds the error \emph{after} the optimization step.} Lipschitz function \\ \qquad ($\sigma=1.170$)\end{tabular}}                                                                  & No. of Subregions   & 256  &  1024  & 4096 \\ 
		& Comp. Time (s) &  57.18  &  214.40   &  786.88 \\ \hline \vspace*{-0.55cm}
	\end{tabular}
	\vskip -0.1in
\end{minipage}
\vspace{-0.25cm}
\end{table}

Table \ref{table:numerical} demonstrates that our proposed abstraction algorithm outperforms the approach in \cite{Alimguzhin2017} because of the tighter bounds $\sigma$ that we can obtain, with the exception of the case when we only assume continuity but not differentiability (i.e., $x \cos y$ is a $C^0$ function). Moreover, the computation (CPU) time is proportional to the resulting number of subregions.

As above-mentioned, the choice of desired accuracy $ \varepsilon_f$ impacts on the number of subregions, where a larger $ \varepsilon_f$ leads to fewer subregions, as shown in Figure \ref{fig:abstraction}. On the other hand, the choice of approximation error bound also impacts the number of subregions, where a tighter bound leads to less subregions, as illustrated in Figure \ref{fig:abstraction2}. 

\begin{figure}[t!]\vspace{-0.2cm}
\begin{center}
		\includegraphics[scale=0.12,trim=30mm 20mm 20mm 30mm,clip]{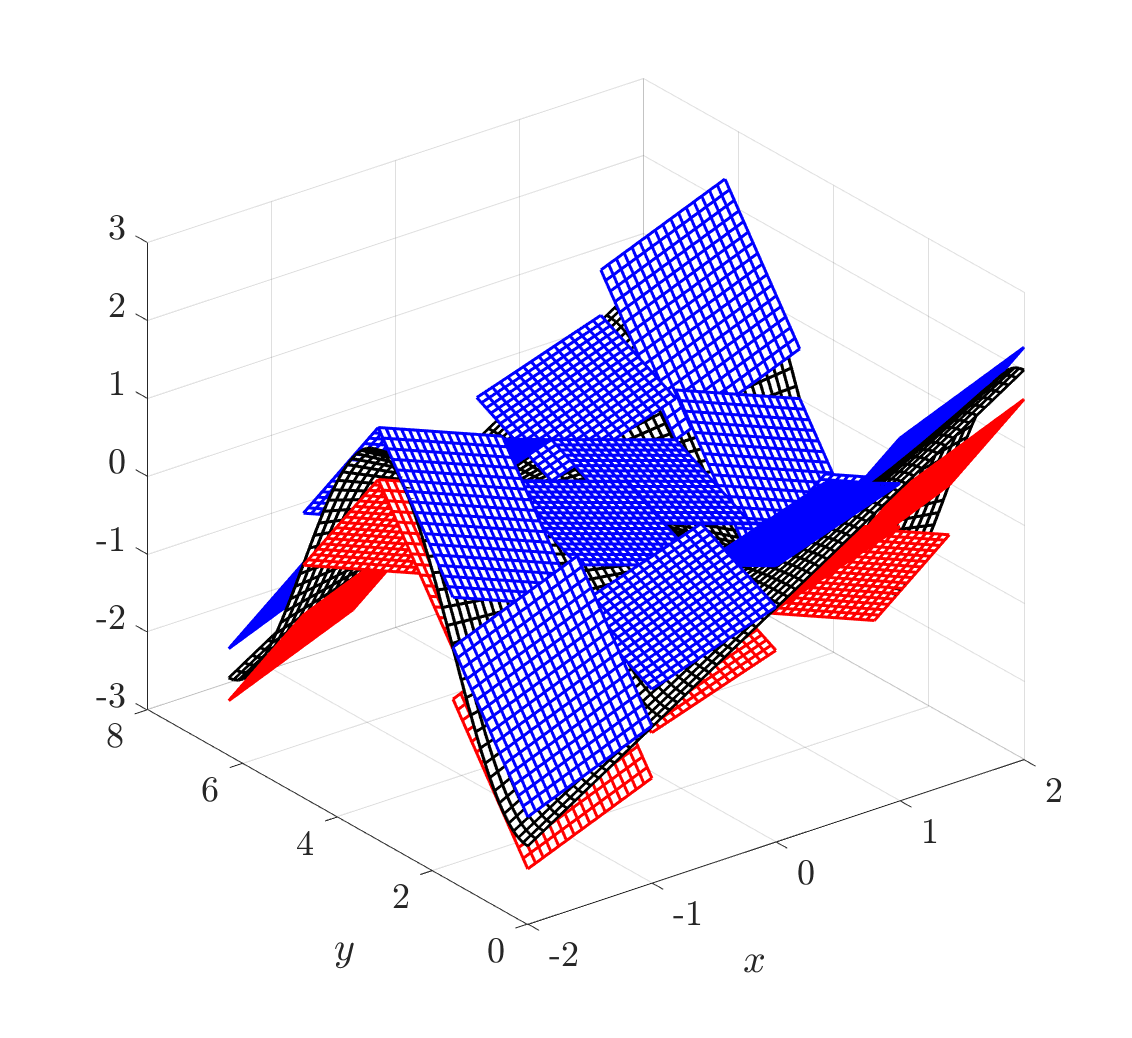}
		\includegraphics[scale=0.12,trim=30mm 20mm 20mm 30mm,clip]{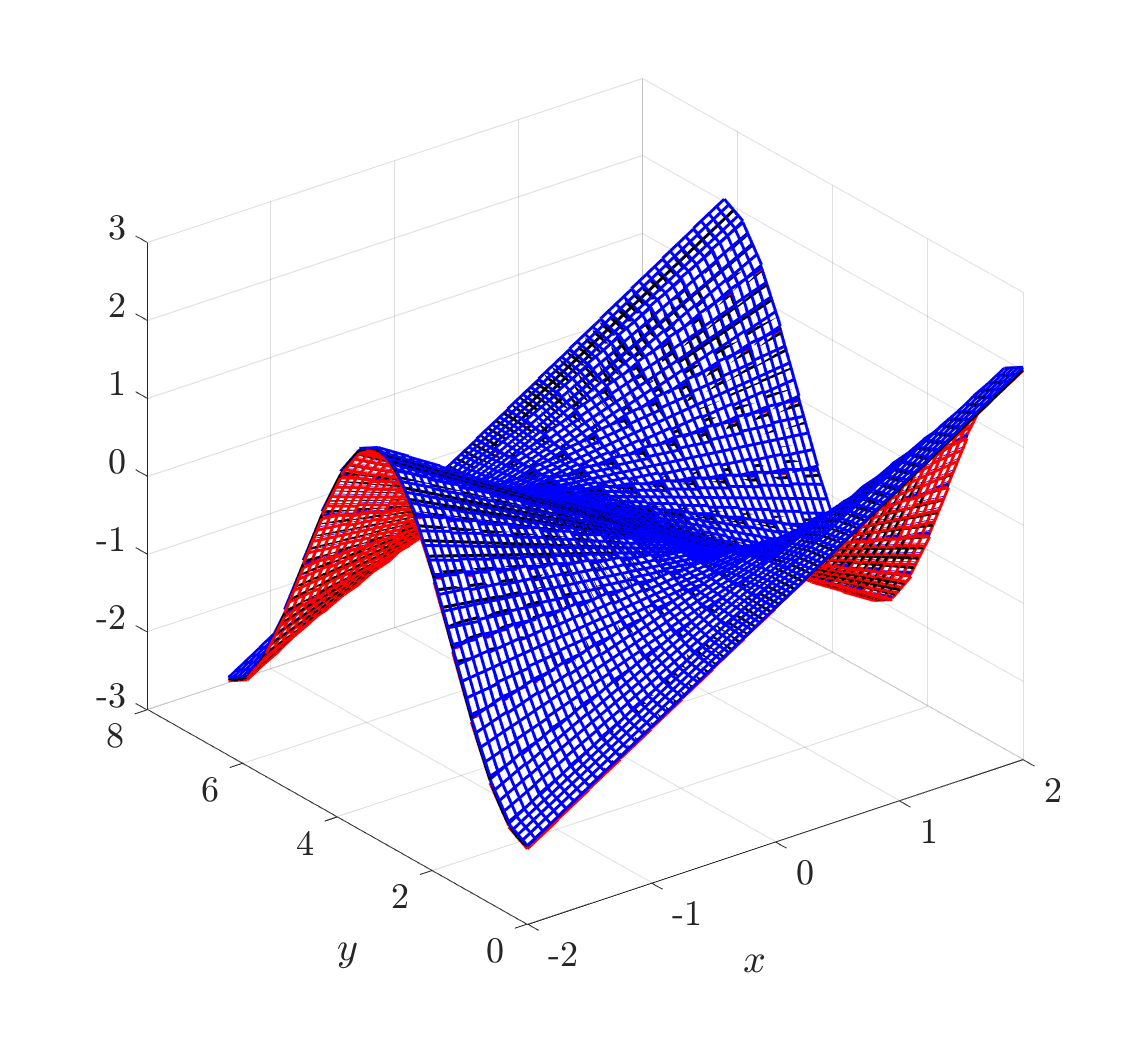}
		\caption{Affine abstraction of $x \cos y$ using an approximation error bound $\sigma=0.228$ and desired accuracies, $ \varepsilon_f=1$ (\textbf{left}) and $ \varepsilon_f=0.05$ (\textbf{right}), resulting in 16 and 256 subregions.\label{fig:abstraction} }
	\end{center}
\end{figure}

\begin{figure}[t!]\vspace{-0.5cm}
\begin{center}
		\includegraphics[scale=0.12,trim=30mm 20mm 20mm 30mm,clip]{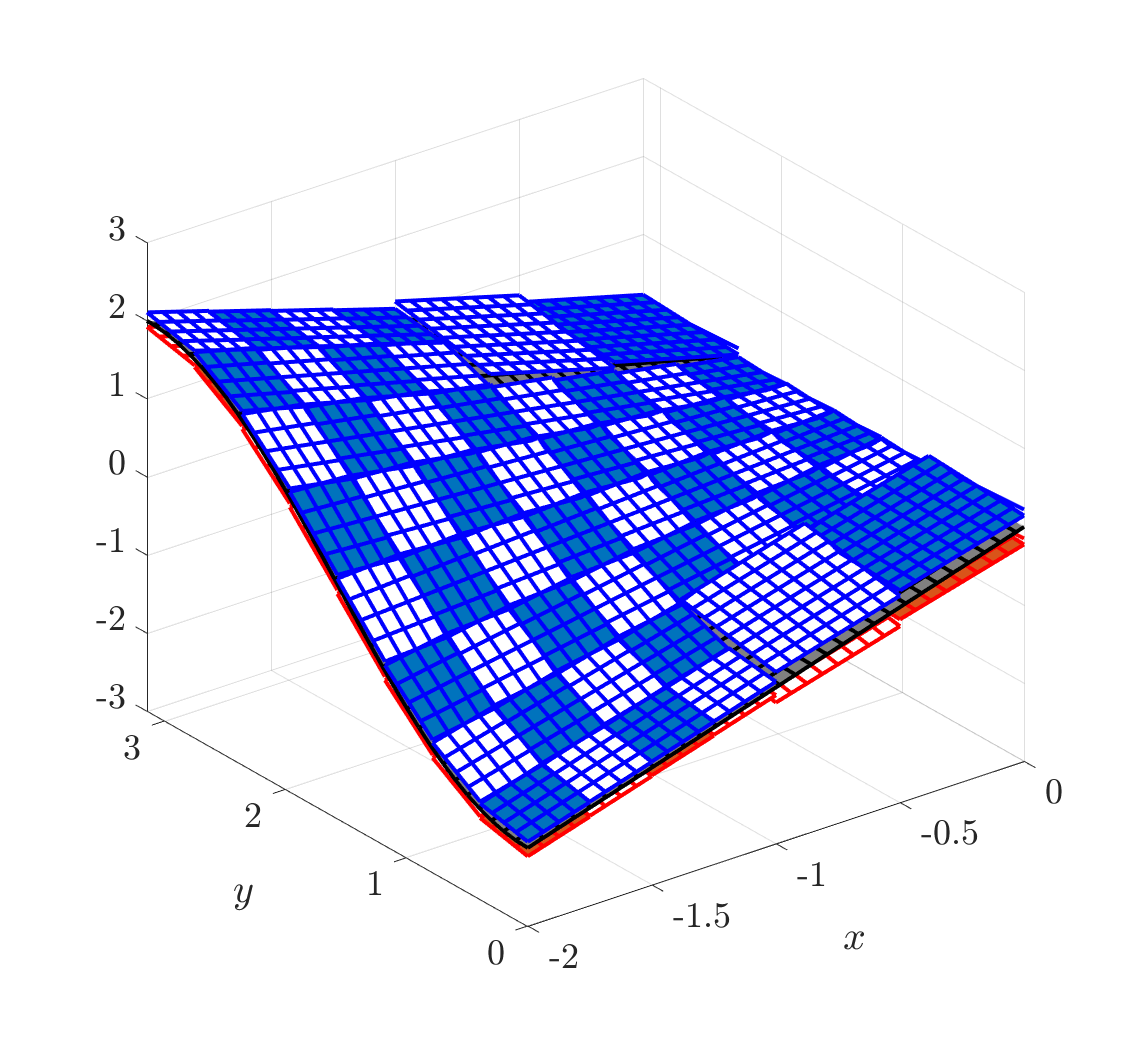}
		\includegraphics[scale=0.12,trim=30mm 20mm 20mm 30mm,clip]{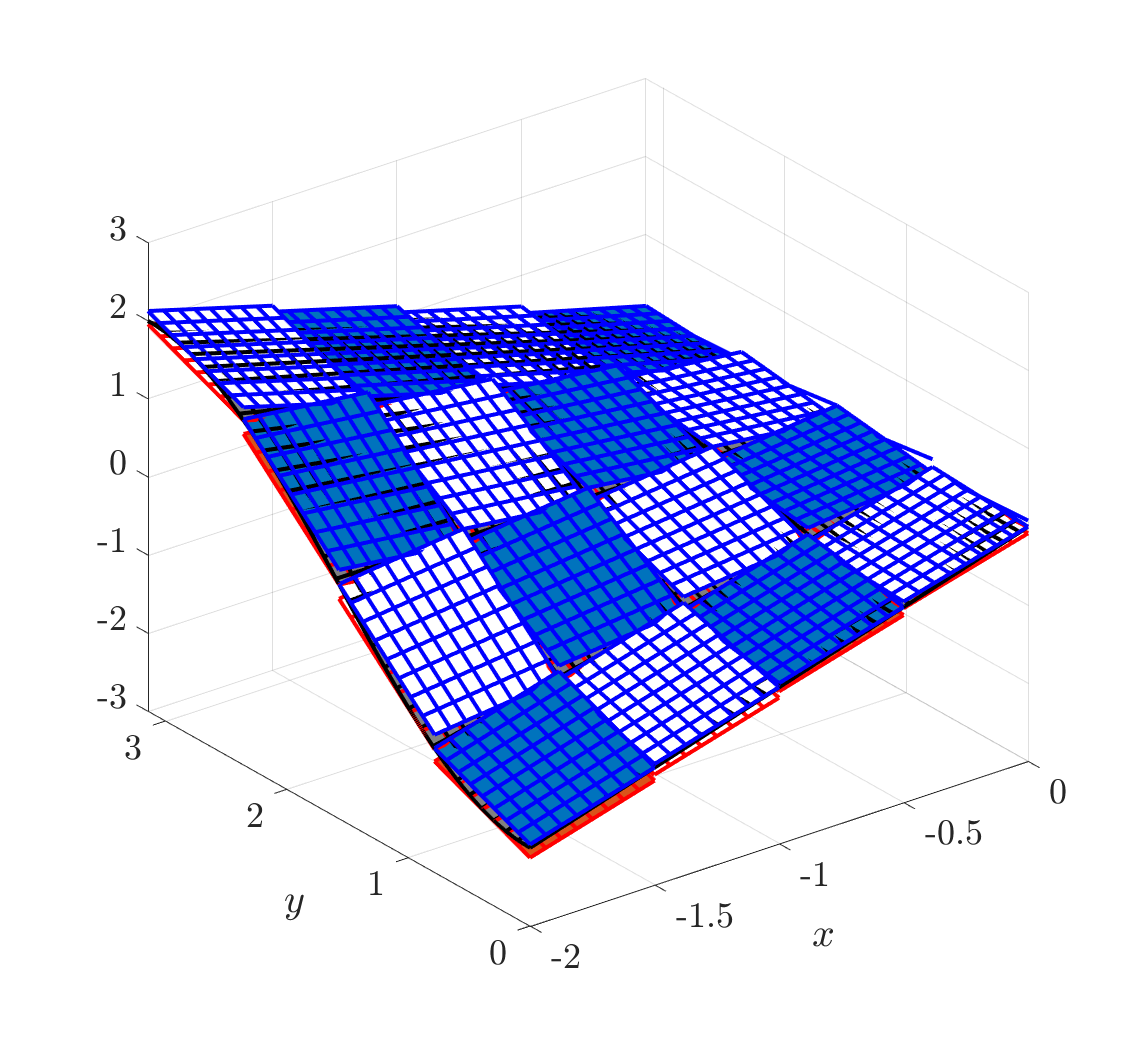}
		\caption{Affine abstraction of $x \cos y$ using a desired accuracy $ \varepsilon_f=0.4$ and approximation error bounds $\sigma=1.170$ (\textbf{left}, \cite{Alimguzhin2017}) and $\sigma=0.228$ (\textbf{right}); zoomed in $[-2,0] \times [0,\pi]$ with added emphasis (colored) on different subregions.\label{fig:abstraction2} }
	\end{center}
\end{figure}

\subsection{Two-Dimensional Dubins Vehicle Dynamics with Application to Active Model Discrimination} \label{sec:2d}

Next, we consider the Dubins vehicles dynamics \cite{Dubins1957} that consist of two functions $$f_1 (v,\phi)= v \cos \phi, \ f_2(v,\phi)=v \sin \phi,$$
where $v$ and $\phi$ are states that represent the speed and heading angle of a vehicle, respectively. As in \cite{Singh2018}, where a mixed-integer nonlinear optimization (MINLP) approach  is used to obtain an affine abstraction, we consider only one region (i.e., without subdividing into subregions) with the speed between $20m/s$ and $30m/s$ (72 to 108 $km/h$) and the heading angle between $-25^\circ$ to $25^\circ$ ($[-0.44,0.44]$ $rad$). Moreover, we consider an objective function that minimizes $\gamma_A \| \overline{A}-\underline{A}\|_\infty +\gamma_h \|\overline{h}-\underline{h}\|_\infty$, where $\gamma_A$ and $\gamma_h$ are chosen as 0.5 and 5, respectively.

Both the proposed mesh-based and the MINLP \footnote{Additional constraints $\overline{A}\geq \underline{A}$ and $\overline{h} \geq \underline{h}$ are imposed on the MINLP-based affine abstraction formulation to ensure that the generally suboptimal MINLP formulation finds a feasible solution.} (see  \cite{Singh2018} for details) approaches are able to obtain affine abstractions of the Dubins dynamics. \yong{For very small resolution $r$, e.g., $r=25$, the mesh-based approach \yong{(polynomial-time)} obtained a worse optimal value than the MINLP approach (independent of resolution\yong{; NP-hard}), however the optimal value decreases rapidly as the resolution is increased, as shown in Figure \ref{fig:abstraction3}.} On the other hand, the computation (CPU) time of the mesh-based approach increases with increasing resolution but is still generally faster than the MINLP approach up until the resolution of over $r=3000$.

\begin{figure}[t!]
\begin{center}
		\includegraphics[scale=0.32,trim=2mm 0mm 0mm 0mm]{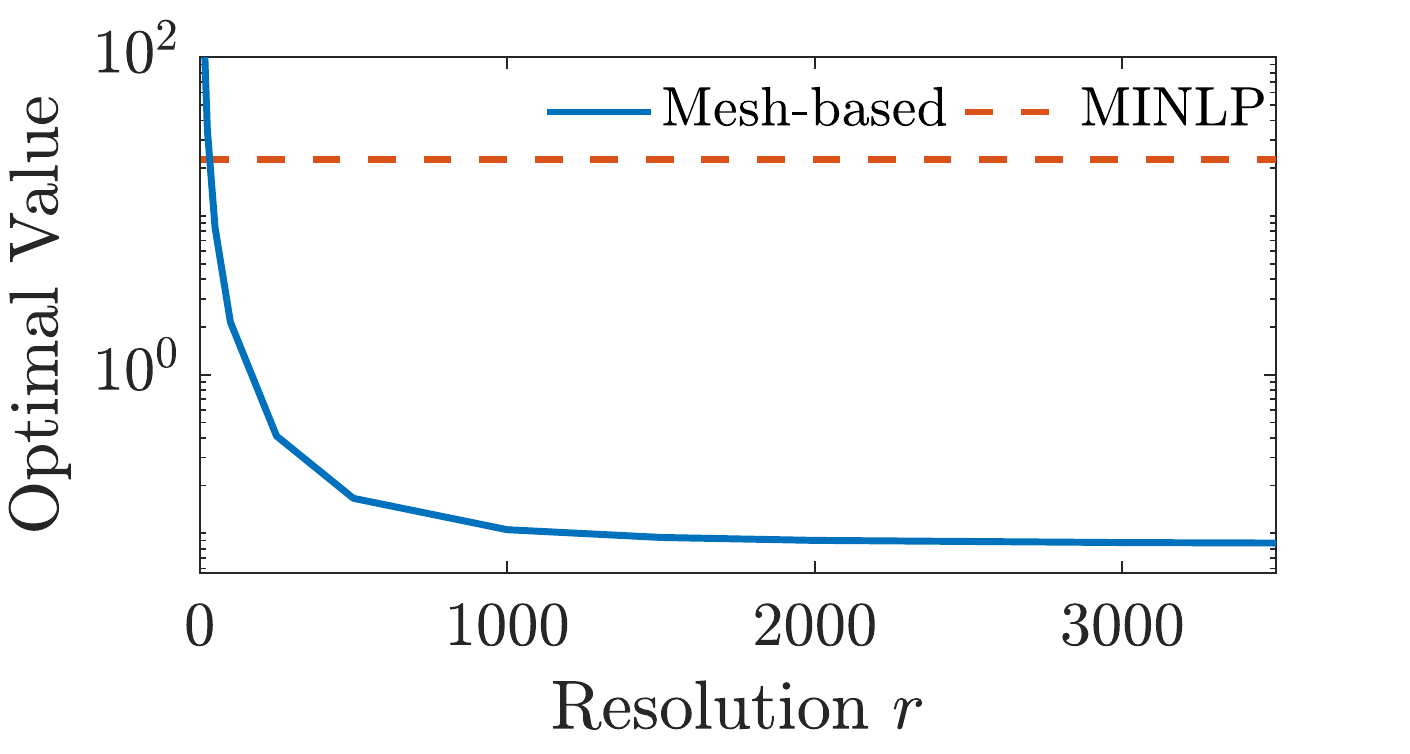}
		\includegraphics[scale=0.32,trim=10mm 0mm 10mm 0mm]{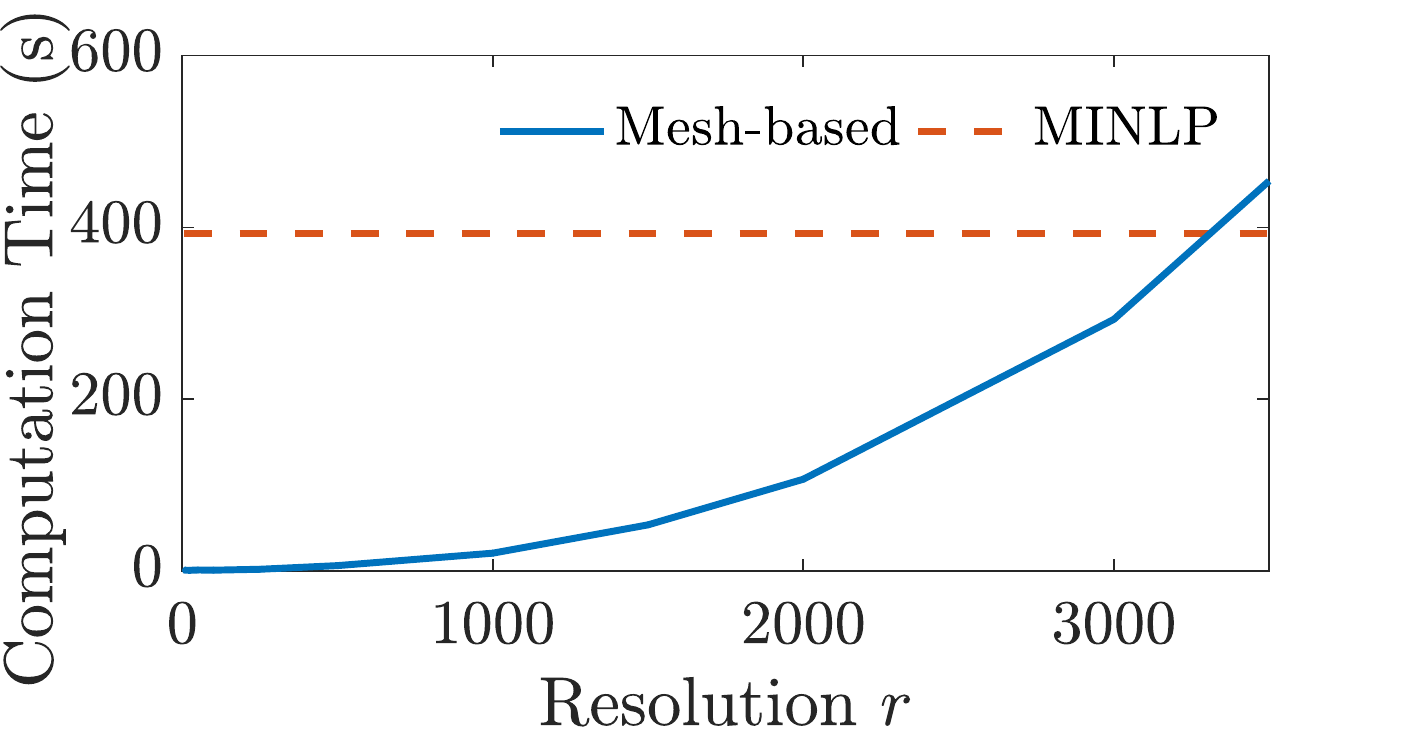}
		\caption{Decreasing optimal value (\textbf{left}) and increasing computation (CPU) time (\textbf{right}) as the resolution $r$ is increased for the mesh-based affine abstraction approach, in comparison with the values obtained from the MINLP-based approach.\label{fig:abstraction3} }\vspace{-0.3cm}
	\end{center}
\end{figure}

In addition, to illustrate the applicability of our proposed affine abstraction to estimator designs, we apply the obtained affine abstractions to solve the active model discrimination problem for identifying driver intention (without considering the ego car's responsibility, for simplicity; see \cite[Section 6]{Singh2018} for detailed models and notations), where the goal is to find the minimum input excitation that guarantees that  the different intention models are sufficiently differentiated from each other. Table \ref{table:numerical2} shows that the optimal values ($\{1,2,\infty\}$-norms of the excitation input $u_T$)   that are obtained for the active model discrimination problem based on mesh-based affine abstraction are lower than when the MINLP-based abstraction is used, as desired. However, this improvement comes at the cost of higher computation times.

\begin{table}[ht] \vspace{-0.15cm}
	\centering
	\caption{Optimal values and computation (CPU) times for active model discrimination when using affine abstractions of Dubins dynamics from MINLP- and mesh-based approaches.} 
	\label{table:numerical2}
	\begin{tabular}{c c c c c}
		\hline
		& & $\|u_T\|_1$ & $\|u_T\|_\infty$ & $\|u_T\|_2$\\ \hline
	\multirow{2}{*}{\begin{tabular}[c]{@{}c@{}}MINLP-based \cite{Singh2018}\end{tabular} }                                                                  & Optimal Value   & 1.0819  &  0.4930 & 0.4171 \\ 
		& CPU Time (s) & 10.2590  &  8.6219 & 277.7508 \\ \hline
			\multirow{2}{*}{\begin{tabular}[c]{@{}c@{}}Mesh-based \\ ($r=3500$)\end{tabular} }                                                                  & Optimal Value   & 0.4341  &  0.2937 &  0.1746 \\ 
		& CPU Time (s) & 9.4357  & 27.4894 & 1600.5861  \\ \hline
	\end{tabular}
\end{table}\vspace{-0.15cm}


\section{Conclusion}
This paper presents a piecewise affine abstraction approach for nonlinear systems using tighter interpolation bounds. 
We divide the domain of interest into smaller subregions that form a cover of the domain with a desired approximation accuracy for each subregion. 
On each subregion, the nonlinear dynamics is conservatively approximated by a pair of piecewise affine functions, which brackets the original nonlinear dynamics. 
Our novel analysis allows for the use of tighter interpolation bounds, thus the proposed abstraction method achieves better time efficiency and requires less subregions for the same desired approximation accuracy when compared to existing approaches. Our method also applies to nonlinear functions with different degree of smoothness. 
We demonstrated the advantages of our approach in simulation and illustrated its applicability for the problem of active model discrimination. 
Future works will explore partitioning the domain of interest into subregions with a non-uniform, non-rectangular mesh, e.g., simplicial mesh, with the objective of improving the approximation quality and accuracy.

\bibliographystyle{IEEEtran}
\bibliography{biblio}

\begin{thebibliography}{10}
\providecommand{\url}[1]{#1}
\csname url@rmstyle\endcsname
\providecommand{\newblock}{\relax}
\providecommand{\bibinfo}[2]{#2}
\providecommand\BIBentrySTDinterwordspacing{\spaceskip=0pt\relax}
\providecommand\BIBentryALTinterwordstretchfactor{4}
\providecommand\BIBentryALTinterwordspacing{\spaceskip=\fontdimen2\font plus
\BIBentryALTinterwordstretchfactor\fontdimen3\font minus
  \fontdimen4\font\relax}
\providecommand\BIBforeignlanguage[2]{{%
\expandafter\ifx\csname l@#1\endcsname\relax
\typeout{** WARNING: IEEEtran.bst: No hyphenation pattern has been}%
\typeout{** loaded for the language `#1'. Using the pattern for}%
\typeout{** the default language instead.}%
\else
\language=\csname l@#1\endcsname
\fi
#2}}

\bibitem{Tabuada2009}
P.~Tabuada, \emph{Verification and control of hybrid systems: a symbolic
  approach}.\hskip 1em plus 0.5em minus 0.4em\relax Springer, 2009.

\bibitem{althoff2008reachability}
M.~Althoff, O.~Stursberg, and M.~Buss, ``Reachability analysis of nonlinear
  systems with uncertain parameters using conservative linearization,'' in
  \emph{IEEE Conference on Decision and Control}, 2008, pp. 4042--4048.

\bibitem{Girard2012}
A.~Girard and S.~Martin, ``Synthesis for constrained nonlinear systems using
  hybridization and robust controller on symplices,'' \emph{IEEE Transactions
  on Automatic Control}, vol.~57, no.~4, pp. 1046--1051, 2012.

\bibitem{Alimguzhin2017}
V.~Alimguzhin, F.~Mari, I.~Melatti, I.~Salvo, and E.~Tronci, ``Linearizing
  discrete-time hybrid systems,'' \emph{IEEE Transactions on Automatic
  Control}, vol.~62, no.~10, pp. 5357--5364, 2017.

\bibitem{Singh2018}
K.~Singh, Y.~Ding, N.~Ozay, and S.~Z. Yong, ``Input design for nonlinear model
  discrimination via affine abstraction,'' in \emph{IFAC Conference on Analysis
  and Design of Hybrid Systems}, 2018, pp. 1--8, accepted.

\bibitem{asarin2007hybridization}
E.~Asarin, T.~Dang, and A.~Girard, ``Hybridization methods for the analysis of
  nonlinear systems,'' \emph{Acta Informatica}, vol.~43, no.~7, pp. 451--476,
  2007.

\bibitem{Asarin2003}
------, ``Reachability analysis of nonlinear systems using conservative
  approximation,'' in \emph{Int. Workshop on Hybrid Systems: Computation and
  Control}.\hskip 1em plus 0.5em minus 0.4em\relax Springer, 2003, pp. 20--35.

\bibitem{Azuma2010}
S.-i. Azuma, J.-i. Imura, and T.~Sugie, ``Lebesgue piecewise affine
  approximation of nonlinear systems,'' \emph{Nonlinear Analysis: Hybrid
  Systems}, vol.~4, no.~1, pp. 92--102, 2010.

\bibitem{Han2006}
Z.~Han and B.~H. Krogh, ``Reachability analysis of nonlinear systems using
  trajectory piecewise linearized models,'' in \emph{the American Control
  Conference}, 2006, pp. 1505--1510.

\bibitem{Bak2016}
S.~Bak, S.~Bogomolov, T.~A. Henzinger, T.~T. Johnson, and P.~Pradyot,
  ``Scalable static hybridization methods for analysis of nonlinear systems,''
  in \emph{the 19th International Conference on Hybrid Systems: Computation and
  Control}.\hskip 1em plus 0.5em minus 0.4em\relax Springer, 2016, pp.
  155--164.

\bibitem{Ramdani2009}
N.~Ramdani, N.~Meslem, and Y.~Candau, ``A hybrid bounding method for computing
  an over-approximation for the reachable set of uncertain nonlinear systems,''
  \emph{IEEE Transactions on Automatic Control}, vol.~54, no.~10, pp.
  2352--52\,364, 2009.

\bibitem{Stampfle2000}
M.~St\"{a}mpfle, ``Optimal estimates for the linear interpolation error for
  simplices,'' \emph{Journal of Approximation Theory}, vol. 103, pp. 78--90,
  2000.

\bibitem{dang2010accurate}
T.~Dang, O.~Maler, and R.~Testylier, ``Accurate hybridization of nonlinear
  systems,'' in \emph{ACM International Conference on Hybrid Systems:
  Computation and Control}, 2010, pp. 11--20.

\bibitem{Kuhn1960}
H.~W. Kuhn, ``Some combinatorial lemmas on topology,'' \emph{IBM Journal of
  Research and Development}, vol.~4, no.~5, pp. 518--524, 1960.

\bibitem{mladineo1986algorithm}
R.~H. Mladineo, ``An algorithm for finding the global maximum of a multimodal,
  multivariate function,'' \emph{Mathematical Programming}, vol.~34, no.~2, pp.
  188--200, 1986.

\bibitem{Dubins1957}
L.~E. Dubins, ``On curves of minimal length with a constraint on average
  curvature, and with prescribed initial and terminal positions and tangents,''
  \emph{American Journal of Mathematics}, vol.~79, pp. 497--516, 1957.

\end{thebibliography}

\end{document}